\documentclass[12pt]{amsart}
\usepackage{times}
\usepackage{amssymb}
\usepackage{amsmath}
\usepackage{color}
\usepackage{comment}
\usepackage{graphicx}
\usepackage[all]{xy}
\usepackage{tikz}
\usepackage{psfrag,xmpmulti,amscd,color,pstricks, import,enumerate,soul}
\usepackage[normalem]{ulem}
\begin{document}
\numberwithin{equation}{section}

\def\1#1{\overline{#1}}
\def\2#1{\widetilde{#1}}
\def\3#1{\widehat{#1}}
\def\4#1{\mathbb{#1}}
\def\5#1{\frak{#1}}
\def\6#1{{\mathcal{#1}}}

\newcommand{\de}{\partial}
\newcommand{\R}{\mathbb R}
\newcommand{\al}{\alpha}
\newcommand{\tr}{\widetilde{\rho}}
\newcommand{\tz}{\widetilde{\zeta}}
\newcommand{\tv}{\widetilde{\varphi}}
\newcommand{\tO}{\widetilde{\Omega}}
\newcommand{\hv}{\hat{\varphi}}
\newcommand{\tu}{\tilde{u}}
\newcommand{\usc}{{\sf usc}}
\newcommand{\tF}{\tilde{F}}
\newcommand{\debar}{\overline{\de}}
\newcommand{\Z}{\mathbb Z}
\newcommand{\C}{\mathbb C}
\newcommand{\Po}{\mathbb P}
\newcommand{\zbar}{\overline{z}}
\newcommand{\G}{\mathcal{G}}
\newcommand{\So}{\mathcal{U}}
\newcommand{\Ko}{\mathcal{K}}
\newcommand{\U}{\mathcal{U}}
\newcommand{\B}{\mathbb B}
\newcommand{\oB}{\overline{\mathbb B}}
\newcommand{\Cur}{\mathcal D}
\newcommand{\Dis}{\mathcal Dis}
\newcommand{\Levi}{\mathcal L}
\newcommand{\SP}{\mathcal SP}
\newcommand{\Sp}{\mathcal Q}
\newcommand{\Ma}{\mathcal M}
\newcommand{\Co}{\mathcal C}
\newcommand{\Hol}{{\sf Hol}(\mathbb H, \mathbb C)}
\newcommand{\Aut}{{\sf Aut}(\mathbb D)}
\newcommand{\D}{\mathbb D}
\newcommand{\oD}{\overline{\mathbb D}}
\newcommand{\oX}{\overline{X}}
\newcommand{\loc}{L^1_{\rm{loc}}}
\newcommand{\loci}{L^\infty_{\rm{loc}}}
\newcommand{\la}{\langle}
\newcommand{\ra}{\rangle}
\newcommand{\thh}{\tilde{h}}
\newcommand{\N}{\mathbb N}
\newcommand{\kd}{\kappa_D}
\newcommand{\Hr}{\mathbb H}
\newcommand{\ps}{{\sf Psh}}
\newcommand{\tg}{\widetilde{\gamma}}

\newcommand{\subh}{{\sf subh}}
\newcommand{\harm}{{\sf harm}}
\newcommand{\ph}{{\sf Ph}}
\newcommand{\tl}{\tilde{\lambda}}
\newcommand{\ts}{\tilde{\sigma}}

\def\v{\varphi}
\def\Re{{\sf Re}\,}
\def\Im{{\sf Im}\,}

\def\e{{\sf e}}

\def\dist{{\rm dist}}
\def\const{{\rm const}}
\def\rk{{\rm rank\,}}
\def\id{{\sf id}}
\def\aut{{\sf aut}}
\def\Aut{{\sf Aut}}
\def\CR{{\rm CR}}
\def\GL{{\sf GL}}
\def\U{{\sf U}}

\def\la{\langle}
\def\ra{\rangle}
\def\Ha{\mathbb H}

\newtheorem{theorem}{Theorem}[section]
\newtheorem{lemma}[theorem]{Lemma}
\newtheorem{proposition}[theorem]{Proposition}
\newtheorem{corollary}[theorem]{Corollary}

\theoremstyle{definition}
\newtheorem{definition}[theorem]{Definition}
\newtheorem{defn}[theorem]{Definition}
\newtheorem{example}[theorem]{Example}

\theoremstyle{remark}
\newtheorem{remark}[theorem]{Remark}
\newtheorem{rem}[theorem]{Remark}
\numberwithin{equation}{section}

\definecolor{OrangeRed}{cmyk}{0,0.6,1,0}   
\definecolor{DarkBlue}{cmyk}{1,1,0,0.20}
\definecolor{DarkGreen}{cmyk}{1,0,0.6,0.2}
\definecolor{myblue}{rgb}{0.66,0.78,1.00}
\definecolor{Violet}{cmyk}{0.79,0.88,0,0}
\definecolor{Lavender}{cmyk}{0,0.48,0,0}
\newcommand{\violet }{\color{Violet}}
\newcommand{\lavender }{\color{Lavender}}
\renewcommand{\green}{\color{DarkGreen}}
\renewcommand{\red}{\color{OrangeRed}}

\def\re{{\sf Re}\,}
\def\im{{\sf Im}\,}
\renewcommand{\ra}{\rightarrow}
\newcommand{\ov}{\overline}
\newcommand{\NN}{\mathcal{N}}
\renewcommand{\SS}{\mathcal{S}}
\renewcommand{\H}{\mathbb{H}}
\renewcommand{\emptyset}{\varnothing}

\makeatletter
\@namedef{subjclassname@2020}{\textup{2020} Mathematics Subject Classification}
\makeatother

\title[Denjoy-Wolff Theorem]{The Denjoy-Wolff Theorem in simply connected domains}
\author[A. M. Benini]{Anna Miriam Benini$^{\dag}$}
\author[F. Bracci]{Filippo Bracci$^{\dag\dag}$}

\address{A. M. Benini: Dipartimento di Scienze Matematiche, Fisiche e Informatiche\\ Parco Area delle Scienze 53/A (Sciences and Technology Campus) \\ 43124 Parma, Italy } \email{annamiriam.benini@unipr.it		}

\address{F. Bracci: Dipartimento Di Matematica\\
Universit\`{a} di Roma \textquotedblleft Tor Vergata\textquotedblright\ \\
Via Della Ricerca Scientifica 1\\ 00133 
Roma, Italy} \email{fbracci@mat.uniroma2.it}


\thanks{$\dag$ Partially supported  by the French Italian University and Campus France through the Galileo program, under the project {\sl From rational to transcendental: complex dynamics and parameter spaces and by GNAMPA, INdAM}.\\
$\dag\dag$ Partially supported by  the MIUR Excellence Department Project 2023-2027 MatMod@Tov awarded to the
Department of Mathematics, University of Rome Tor Vergata. \\
Both authors are also partially supported by PRIN {\sl Real and Complex
Manifolds: Topology, Geometry and holomorphic dynamics} n.2017JZ2SW5 and by GNSAGA of INdAM.}

\begin{abstract}
We characterize the simply connected domains $\Omega\subsetneq\C$ that exhibit the Denjoy-Wolff Property, meaning that every holomorphic self-map of $\Omega$ without fixed points has a Denjoy-Wolff point. We demonstrate that this property holds if and only if every automorphism of $\Omega$ without fixed points in $\Omega$ has a Denjoy-Wolff point. Furthermore, we establish that the Denjoy-Wolff Property is equivalent to the existence of what we term an ``$H$-limit''  at each boundary point for  a Riemann map associated with the domain. The $H$-limit condition is stronger than the existence of non-tangential limits but weaker than unrestricted limits.

As an additional result of our work, we prove that there exist bounded simply connected domains where the Denjoy-Wolff Property holds but which are not visible in the sense of Bharali and Zimmer. Since visibility is a sufficient condition for the Denjoy-Wolff Property, this proves that in general it is not necessary.
\end{abstract}

\maketitle

\tableofcontents

\section{Introduction and statement of the results}

Let $\D$ be the unit disc. A celebrated theorem of Denjoy-Wolff \cite{Denjoy,Wolff1, Wolff2} states that
\begin{theorem}[Denjoy-Wolff] Let $f:\D\ra\D$ be holomorphic without fixed points in $\D$. Then, there exists a (unique) point $\tau\in\partial\D$, called {\sl Denjoy-Wolff point}, such that the sequence of iterates $\{f^{\circ n}\}$ converges to $\tau$ uniformly on compacta.
\end{theorem}
The Denjoy-Wolff theorem has been greatly generalized to several other contexts, in one and several variables and in Gromov hyperbolic metrix spaces, both for its intrinsic interest and several applications (see for instance \cite{AbateTaut, AbateNewbook, BCDbook, Ka} and references therein for more details). 

In this paper, we address the problem of identifying which simply connected domains satisfy a Denjoy-Wolff type theorem.

If a simply connected domain $D\subsetneq\C$ has locally connected boundary, then by Carath\'eodory's extension theorem every Riemann map $h:\D\to D$ extends continuously up to $\partial\D$. Hence, if $f:D\to D$ is holomorphic without fixed points, by the Denjoy-Wolff Theorem the iterates of $g:=h^{-1}\circ f\circ h:\D\ra \D$ converges to some point $\tau\in\partial\D$ and then the iterates of $f$ converge to $h(\tau)$. Therefore in simply connected domains with locally connected boundary, a Denjoy-Wolff type theorem holds.

On the other hand, if $D\subsetneq\C$ is any simply connected domain which  has a prime end $\underline{\xi}$ whose principal part is not a point,  $f:\D\ra\D$ is any holomorphic self-map of $\D$ with  Denjoy-Wolff  point $\tau\in\partial\D$, and $h:\D\ra D$ is a Riemann map which makes corresponding $\tau$ to the prime end $\underline{\xi}$, then  (see Section~\ref{sub:easy})  the holomorphic map $g:=h\circ f\circ h^{-1}:D\ra D$ has the property that for every $z\in D$ the cluster set of $\{g^{\circ n}(z)\}$ contains a continuum of points.  

The two previous classes of examples show that simply connected domains can or can not satisfy a Denjoy-Wolff type theorem. In this paper, we find  sharp conditions for this. In order to state our main result, we need a few definitions.

We denote by $\C_\infty$ the Riemann sphere. If $\Omega\subset \C$ is a domain, we let $\partial \Omega$  be its  boundary in $\C_\infty$ and $\overline{\Omega}$    its  closure   in $\C_\infty$---in particular, if $\Omega$ is bounded, $\partial \Omega$  is the Euclidean  boundary of $\Omega$ and $\overline{\Omega}$ its Euclidean closure. Also, we let
\begin{equation*}
\begin{split}
\hbox{Hol}_d(\Omega,\Omega)&:=\{f:\Omega\to \Omega : \hbox{ $f$ is holomorphic and without fixed points in $\Omega$}\}.\\
\hbox{Aut}_d(\Omega)&:=\{ f\in\hbox{Hol}_d(\Omega,\Omega)\hbox{: $f$ is an automorphism}\}\\
\end{split}
\end{equation*}

With these notations at hand, we give the following definition:

\begin{defn}[Denjoy-Wolff Property] 
Let $\Omega\subset \C$  be a simply connected domain. We say that  $f\in \hbox{Hol}_d(\Omega,\Omega)$  has {\sl Denjoy-Wolff point} if there exists $\tau\in\partial\Omega$ such that the sequence of iterates $\{f^{\circ n}\}$ converges to $\tau$ uniformly on compacta. The  domain $\Omega\subset \C$ has the {\sl Denjoy-Wolff Property} if every $f\in \hbox{Hol}_d(\Omega,\Omega)$ has  Denjoy-Wolff point.
\end{defn}

Now we define a notion of boundary limit for holomorphic function defined on $\D$, that we call $H$-limit.  Given $\sigma\in\partial\D$ and $R>0$  the {\sl horosphere of center $\sigma$ and radius $R$} is given by:
\begin{equation}\label{eq:horospheres in D}
E(\sigma,R):=\{z\in \D: \frac{|\sigma-z|^2}{1-|z|^2}<R\}.
\end{equation}

\begin{defn}[H-limits]\label{Def:H-limits}
Let $h:\D\to \C$ be a function. We say that $L\in \C_\infty$ is the {\sl $H$-limit} of $h$ at $\sigma\in \partial\D$---for short, $H$-$\lim_{z\to\sigma}h(z)=L$---if for every sequence $\{z_n\}$ converging to $\sigma$ such that $\{z_n\}$ is eventually contained in $E(\sigma, R)$ for some $R>0$ then $\{h(z_n)\}$ converges to $L$.
\end{defn}
In somewhat  different forms (usually considering sequences which are eventually contained in {\sl all} horospheres), or in an implicit way (especially related to the so-called Julia's Lemma) {\sl $H$-limits}  already appeared in the literature (see, {\sl e.g.}, \cite{AbateTaut} and \cite[Section 8]{BH}). 

 If a holomorphic function $h:\D\to \C$ has $H$-limit $L$ at $\sigma\in\partial\D$ then it has non-tangential limit $L$ at $\sigma$. While, if $h$ has (unrestricted) limit $L$ at $\sigma$ then $H$-$\lim_{z\to \sigma}h(z)=L$.
 
 Finally, a {\sl (Busemann) horosphere} of a simply connected domain properly contained in $\C$ is the image of a horosphere of $\D$ via a Riemann map of the domain (see Subsection~\ref{sec:hyp} for details). With these definitions at hand, we can state our main result:

\begin{theorem}\label{Thm:main Intro}
Let $\Omega\subsetneq \C$ be a simply connected domain. The following are equivalent:
\begin{enumerate}
\item $\Omega$ has the Denjoy-Wolff Property, 
\item every $g\in \hbox{Aut}_d(\Omega)$ has a Denjoy-Wolff point,
\item the closure in $\C_\infty$ of every (Busemann) horosphere of $\Omega$ intersects $\partial \Omega$ at one point,
\item every Riemann map $h:\D\to \Omega$ has $H$-limit at every $\sigma\in\partial\D$.
\end{enumerate}
In particular, if  one---and hence all---of the previous conditions holds, then the principal part of every prime end of $\Omega$ is a single point. 
\end{theorem}

Theorem~\ref{Thm:main Intro} is a consequence of Theorem~\ref{Thm:main}, which also provides additional details. For example, in statement (2), it is sufficient to focus solely on ``parabolic'' automorphisms, and in statement (3), it suffices to consider Busemann horospheres of a fixed radius. 

Condition (4) in Theorem~\ref{Thm:main Intro}   is sharp. That is to say, the condition of having $H$-limit is not equivalent to either having unrestricted limits or to having non-tangential limits.  Indeed (see Proposition~\ref{prop:Carsten}), there exists a univalent map $g$ defined on $\D$ which has non-tangential limit at every point of $\partial\D$ but fails to  have  $H$-limit at some boundary point, and, thus, $g(\D)$ does not have the Denjoy-Wolff Property although the principal part of every prime end is a singleton. 

On the other hand, by Caratheordory's extension theorem, a univalent map on $\D$ has unrestricted limit at every point of $\partial \D$ if and only if the boundary of $f(\D)$ is locally connected. We prove the following:
\begin{proposition}\label{Thm:starlike}
There exists a bounded simply connected domain $D\subset \C$ which has the Denjoy Wolff property but whose boundary is not locally connected. In particular   any Riemann map $h:\D\to D$ does not extend continuously to $\partial \D$,  but has $H$-limit (and hence non-tangential limit) at every point of $\partial \D$.
\end{proposition}

As shown in Theorem~\ref{Thm:main Intro}, if a simply connected domain has the Denjoy-Wolff Property, then the principal part of each prime end is a singleton, but, as said before, this condition is not sufficient for granting the Denjoy-Wolff Property. Still, in Section~\ref{sub:easy}, we prove that for a certain class of holomorphic self-maps without fixed points---the so-called {\sl hyperbolic} maps---the condition that every prime end is a singleton is  sufficient (see Corollary~\ref{Cor:hyper-DW}).

Another concept strictly related to the Denjoy-Wolff theorem is that of {\sl visibility} (introduced in \cite{BZ} and, in an equivalent form in case of Gromov hyperbolic domains, in \cite{BNT}). Let $\Omega\subsetneq\C$ be a simply connected domain. Recall that a {\sl geodesic segment} in $\Omega$ is a curve $\gamma:[0,R]\to\Omega$, $R\in (0,+\infty)$ which minimizes the hyperbolic length between $\gamma(s)$ and $\gamma(t)$ for every $s,t\in [0,R]$.

A simply connected domain $\Omega\subsetneq\C$ is {\sl visible} if given any two points $p, q\in\partial\Omega$, $p\neq q$ there exists a compact set $K\subset\subset \Omega$ such that for every sequence of geodesic segments  $\{\gamma_n\}$ such that,  whenever  $\gamma_n:[0, R_n]\to \Omega$ has the property that $\{\gamma_n(0)\}$ converges to $p$ and $\{\gamma_n(R_n)\}$ converges to $q$, then $\gamma_n([0,R_n])\cap K\neq\emptyset$ for all $n$ sufficiently large.

By a result of Bharali and Maitra \cite{BM} (which works in general for bounded domains in $\C^N$, $N\geq 1$)  if $\Omega\subset\C$ is a bounded simply connected visible domain then it satisfies condition (1) of Theorem~\ref{Thm:main}---and hence, all the equivalent conditions of Theorem~\ref{Thm:main Intro}. By \cite[Corollary~3.4]{BNT}, a simply connected bounded domain $\Omega$  in $\C$ is visible if and only if $\partial\Omega$ is locally connected. Therefore, the example in  Proposition~\ref{Thm:starlike}  shows that, in general, there are domains that are not visible and yet have the  Denjoy-Wolff Property:
\begin{corollary}
There exist bounded simply connected domains which have the Denjoy-Wolff Property and are not visible.
\end{corollary}

\bigskip

{\bf Acknowledgements} We are thankful to  Carsten Petersen for  kindly bringing to our attention his preprint \cite{PetersenPrimeEnds} which essentially contains the  construction of the example in Proposition~\ref{prop:Carsten}. We also thank Marco Abate for some helpful discussions about the topic of this paper.

\section{Preliminary results}\label{Sec:due}

\subsection{Prime ends and Carath\'eodory compactification} \label{sec:basic definitions}

Let $\Omega\subsetneq\C$ be a simply connected domain.
 One can define a natural compactification of $\Omega$ by considering the Carath\'eodory prime ends compactification (see, {\sl e.g.}, \cite[Chapter~4]{BCDbook} for details). If $\partial_C\Omega$ denotes the set of all prime ends of $\Omega$, one can define the Carath\'eodory topology on $\overline{\Omega}^C:=\Omega\cup \partial_C\Omega$ in order to make this space compact and Hausdorff. 

If $\Omega=\D$, the unit disc in $\C$, then the identity map ${\sf id}_\D: \D\to\D$ extends in a unique way as a homeomorphism ${\sf id}_\D: \overline{\D}^C \to \overline{\D}$. For $\sigma\in\D$ we denote by $\underline{\sigma}:={\sf id}_\D^{-1}(\sigma)$ (in other words, $\underline{\sigma}$ is the only prime end in $\D$ associated with $\sigma$).

Let $h:\D\to\Omega$ be a Riemann map. Then $h$ extends in a unique way as a homeomorphism $\tilde{h}:\overline{\D}^C\to\overline{\Omega}^C$. Thus, $\hat{h}:=\tilde{h} \circ {\sf id}_\D^{-1}:\overline{\D}\to \overline{\Omega}^C$ is a homeomorphism. 

As a consequence of Carath\'eodory's extension theorem,  $\partial\Omega$ is locally connected if and only if the identity map ${\sf id}: \Omega\to \Omega$ extends as a continuous surjective map $\tilde{\sf id}:\overline{\Omega}^C\to \overline{\Omega}$. Such a map is  a homeomorphism if and only if $\partial \Omega$ is a Jordan curve.

At every prime end $\underline{\xi}\in \partial_C\Omega$ one can associate two compact subsets of $\C_\infty$. The {\sl impression} $I(\underline{\xi})$ of $\underline{\xi}$ is the set of all points $p\in \C_\infty$ such that there exists a sequence $\{z_n\}\subset\Omega$ converging to $\underline{\xi}$ in the Carath\'eodory topology of $\Omega$ and to $p$ in the topology of $\C_\infty$. The {\sl principal part} $\Pi(\underline{\xi})$ of $\underline{\xi}$ is the set of all points $p\in \C_\infty$ such that there exists a null-chain $(C_n)$ representing $\underline{\xi}$ and a sequence $\{z_n\}$ with $z_n\in C_n$ such that $\{z_n\}$ converges to $p$ in $\C_\infty$.

By Carath\'eodory extension theorem (see, {\sl e.g.}, \cite[Prop. 4.4.4]{BCDbook}), for every $\zeta\in\partial \D$  the cluster set $\Gamma(h,\zeta)$ coincides with the impression of the prime end $\hat{f}(\underline{\zeta})$.

\subsection{Hyperbolic distance, horospheres and geodesic rays}\label{sec:hyp}

Let $\Omega\subsetneq\C$ be a simply connected domain. We denote by $\kappa_\Omega$ the {\sl hyperbolic metric} of $\Omega$. This can be defined in several equivalent ways (see, {\sl e.g.}, \cite[Chapter 5]{BCDbook} for details). For any point  $z\in \Omega$  let $h:\D\to\Omega$ be the only biholomorphism (Riemann map) such that $h(0)=z$ and $h'(0)>0$. Consider a vector  $v\in \C$. On the tangent space of $\C$ at $z$ we then define the metric
\[
\kappa_\Omega(z;v)=\frac{\|v\|}{h'(0)},
\]
where $\|v\|$ is the norm of $v$. 
With this normalization, the hyperbolic metric of $\Omega$ has constant curvature $-4$. In particular,
\begin{equation}\label{Eq:Poinc-disc}
\kappa_\D(z)=\frac{|dz|}{1-|z^2|}, 
\end{equation}
while, if $\H:=\{w\in\C: \Re w>0\}$,
\begin{equation}\label{Eq:Poinc-semip}
\kappa_\H(w)=\frac{|dw|}{2\Re w}.
\end{equation}
On simply connected domains we have the following standard inequality (see {\sl e.g.}, \cite[Theorem~5.2.1]{BCDbook}). 
\begin{lemma}[Estimates on hyperbolic metric]\label{lem:hyp metric on simply connected}
Let $\Omega$ be a simply connected domain. Then for any $z\in\Omega$ and $v\in\C$, 
$$
\frac{|v|}{4\dist(z,\partial \Omega)}\leq \kappa_\Omega(z;v)\leq\frac{|v|}{\dist(z,\partial \Omega)},
$$
where $\dist(z,\partial \Omega)$ denotes the Euclidean distance of $z$ from the Euclidean boundary of $\Omega$. 
\end{lemma}

The {\sl hyperbolic length} of an absolutely continuous curve $\gamma:[0,1]\to\Omega$ is
\[
\ell_\Omega(\gamma):=\int_0^1\kappa_\Omega(\gamma(t); \gamma'(t))dt.
\]
Given $z, w\in \Omega$, the {\sl hyperbolic distance} between $z$ and $w$ is
\[
k_\Omega(z,w)=\inf_{\gamma\in \Gamma_{z,w}}\ell_\Omega(\gamma),
\]
where $\Gamma_{z,w}$ is the set of all absolutely continuous curves joining $z$ with $w$. Any biholomorphism between two simply connected domains is an isometry for the corresponding hyperbolic metrics and distances. 

A {\sl geodesic} (for the hyperbolic distance) is an absolutely continuous curve $\gamma:[0,1]\to\Omega$ such that for all $0\leq s\leq t\leq 1$
\[
k_\Omega(\gamma(s),\gamma(t))=\int_s^t \kappa_\Omega(\gamma(r);\gamma'(r))dr.
\]
A {\sl geodesic ray} is an absolutely continuous curve\footnote{In order to avoid burdening the notation in this section we are considering geodesic rays defined in $[0,+\infty)$, but, clearly, the definition is independent of the interval of definition. 
} $\gamma:[0,+\infty)\to\Omega$ such that $k_\Omega(z_0, \gamma(t))\to +\infty$ as $t\to +\infty$ for some---and hence any---$z_0\in\Omega$ and $\gamma|_{[0,T]}$ is a geodesic for all $T>0$. 
It can be proven that every geodesic/geodesic ray in $\Omega$ is a real analytic curve. 

Fix $z_0\in \Omega$. Then there is a one-to-one correspondence among $\partial_C\Omega$ (the prime ends of $\Omega$) and the set of geodesic rays starting from $z_0$. Such a correspondence can be constructed as follows. Let $f:\D\to \Omega$ be a Riemann map such that $f(0)=z_0$. Then $\gamma$ is a geodesic rays starting from $z_0$ if and only if there exists $\sigma\in\partial\D$ such that $\gamma(r)=f(r\sigma)$, $r\in [0,1)$. Hence, the correspondence\footnote{Such a correspondence is also a homeomorphism between the Carath\'eodory boundary of $\Omega$ and the so-called ``Gromov boundary'' of $\Omega$, but we do not need this in here.} is given by associating to $\gamma$ the prime end $\hat{f}(\underline{\sigma})\in \partial_C\Omega$. 

Given  $M\in\R$ and a geodesic ray $\gamma$, we define the {\sl (Busemann) horosphere} 
\begin{equation}\label{Eq:def-oro}
E_\Omega(\gamma, M):=\{z\in \Omega : \lim_{t\to+\infty} [k_\Omega(z, \gamma(t))-k_\Omega(\gamma(0), \gamma(t))]<M\}.
\end{equation}

Let $\D$ be the unit disc and let $R=e^{2M}$ (that is, $M=\frac{1}{2}\log R$). If $h:\D\to \Omega$ is a Riemann map such that $h(0)=\gamma(0)$ then there exists $\sigma\in\partial \D$ such that
\begin{equation}\label{Eq:oro-in-disc-back}
h^{-1}(E_\Omega(\gamma, M))=\{z\in \D: \frac{|\sigma-z|^2}{1-|z|^2}<R\}=E(\sigma,R),
\end{equation}
where $E(\sigma,R)$ is the horosphere defined in \eqref{eq:horospheres in D}.  Notice that for different choices of the Riemann map  $h$, the collection  of Busemann horospheres  will remain the same, but each Busemann horosphere will be associated to  the horosphere centered at a different $\sigma\in\partial \D$. 

\begin{rem}\label{rem:sliding along geodesics}
It follows directly by (\ref{Eq:oro-in-disc-back}), together with the fact that the hyperbolic distance is additive along a geodesic,  that  if $z\in E_\Omega(\gamma, M)$  for some $M$, then the geodesic segment connecting $z$ to $\gamma(0)$ is also  contained in the  same horosphere.  
\end{rem}

The following lemma shows that if two geodesic rays converge to the same prime end,  the corresponding families of horospheres are the same up to changing the parameter $M$.
\begin{lemma}\label{lemma:oro-in-simply-change}
Let $\Omega\subsetneq \C$ be a simply connected domain. Let $\underline{\xi}\in\partial_C \Omega$. If $\gamma, \eta:[0,+\infty)\to\Omega$ are two geodesic rays converging in the Carath\'eodory topology of $\Omega$ to $\underline{\xi}$, then  there exists $T=T(\gamma, \eta)\in\R$ such that for all $M\in\R$
\begin{equation}
\label{Eq:oro-cambia}
E_\Omega(\eta, M)=E_\Omega(\gamma, M+T).
\end{equation} 
\end{lemma}

\begin{proof}
Since Riemann maps extend as homeomorphisms in the Carath\'eodory compactifications, and are isometries for the hyperbolic metric, it is enough to prove the result for $\Omega=\D$. 

Let $\gamma$ be a geodesic ray in $\D$. It is well known  (or, see, {\sl e.g.}, \cite[Chapter~5]{BCDbook}),  that every geodesic ray in $\D$ {\sl lands}, that is, there exists $\sigma\in\partial\D$ such that $\lim_{t\to+\infty}\gamma(t)=\sigma$. Moreover  (see, {\sl e.g.}, \cite[Proposition~1.4.2]{BCDbook}), for every $\sigma\in\partial\D$ and $z\in\D$, 
\begin{equation}\label{eq-limit-oro-eucl}
\lim_{w\to \sigma}[k_\D(z,w)-k_\D(0,w)]=\frac{1}{2}\log\left(\frac{|\sigma-z|^2}{1-|z|^2} \right).
\end{equation}
Note that, since univalent maps are isometries for the hyperbolic distance, \eqref{Eq:oro-in-disc-back} follows immediately from \eqref{eq-limit-oro-eucl}.

Given a geodesic ray $\gamma:[0,+\infty)\to\D$, let 
$$
L(\gamma):=\frac{1}{2}\log\left(\frac{|\sigma-\gamma(0)|^2}{1-|\gamma(0)|^2} \right).
$$
 Hence, by \eqref{eq-limit-oro-eucl}, 
given two geodesic rays $\gamma, \eta:[0,+\infty)\to\D$ such that $\lim_{t\to+\infty}\gamma(t)=\lim_{t\to+\infty}\eta(t)=\sigma\in\partial\D$,   for every $z\in \D$  we have
\begin{equation*}
\begin{split}
\lim_{t\to +\infty}[&k_\D(z,\gamma(t))-k_\D(\gamma(0),\gamma(t))]=
\lim_{t\to +\infty}[k_\D(z,\gamma(t))-k_\D(0,\gamma(t))]-L(\gamma)\\&=\lim_{t\to +\infty}[k_\D(z,\eta(t))-k_\D(0,\eta(t))]-L(\gamma)\\&=\lim_{t\to +\infty}[k_\D(z,\eta(t))-k_\D(\eta(0),\eta(t))]+L(\eta)-L(\gamma).
\end{split}
\end{equation*}
Thus, setting $T=L(\eta)-L(\gamma)$ we are done.
\end{proof}

The following is a remark about symmetric domains.  For $z\in\C$, as usual, we denote its complex conjugate by $\ov{z}$.

\begin{rem}[Symmetric domains] \label{rem:symmetric domains}
Let $\Omega$  be a simply connected domain which is symmetric with respect to the real axis. Then for any $r_0\in \Omega\cap\R$ there is a Riemann map  $h:\Omega\ra D$  such that $h(0)=r_0$ and $h(z)=\ov{h(\ov{z})}$ for any $z\in\D$.  It follows that
\begin{itemize}
\item $\Omega\cap\R$ is connected and it is (the image of) a geodesic in $\Omega$.
\item if $\gamma:[0,1)\to \Omega$ is the geodesic given by $\gamma(t):=h(t)$, then $E_\Omega(\gamma, M)$ is symmetric with respect to the real axis for all $M\in\R$, that is, $z\in E_\Omega(\gamma, M)$ if and only if $\overline{z}\in E_\Omega(\gamma, M)$.
\end{itemize}
The first statement can be found, {\sl e.g.}, in  \cite[Proposition~6.1.3]{BCDbook} (see also its proof). As for the second statement, by \eqref{Eq:def-oro},  $E_\Omega(\gamma, M)=h(E(1,R))$ for some suitable $R>0$. Let $w\in E(1,R)$. Hence,  $\overline{w}\in E(1,R)$. Therefore,  $\overline{h(w)}=h(\overline{w})\in E_\Omega(\gamma, M)$.
\end{rem}

\subsection{The divergence rate and parabolic automorphisms of simply connected domains}\label{subsect:divergence}

Let $\Omega\subsetneq \C$ be a simply connected domain. Let $g:\Omega\to\Omega$ be holomorphic. In \cite{ArBr, BCDbook} it has been proved that the following limit exists (finite) and does not depend on $z_0\in\Omega$:
\[
c(g):=\lim_{n\to\infty} \frac{k_\Omega(g^{\circ n}(z_0), z_0)}{n}.
\]
The number $c(g)\in [0,+\infty)$ is called the {\sl divergence rate} of $g$.

In case $g\in \hbox{Aut}_d(\Omega)$, we say that $g$ is {\sl parabolic} provided $c(g)=0$. 

If $h:\D\to \Omega$ is a Riemann map, then $g$ is a parabolic automorphism of $\Omega$ if and only if there exists an automorphism $\phi$ of $\D$ such that $g=h \circ \phi \circ  h^{-1}$ and such that $\phi$ is a parabolic automorphism of $\D$, {\sl i.e.}, $\phi$ is a linear fractional self-map of $\D$, $\phi(\D)=\D$ , $\phi$ is not the identity and there exists $\tau\in\partial\D$ such that $\phi(\tau)=\tau$ and $\phi'(\tau)=1$.

\subsection{Unrestricted limits, $H$-limits and non-tangential limits} \label{sec:limits}
Let $h:\D\to \C$ be a holomorphic function.  As customary, we say that a sequence $\{z_n\}\subset \D$ converges {\sl non-tangentially}  to $\sigma\in\partial \D$ if $|\sigma-z_n|\geq c (1-|z_n|)$ for some $c>0$ and for all $n$.

Recall the definition of horosphere $E(\sigma, R)$ as in \eqref{eq:horospheres in D}. We consider the following cluster sets: 
\begin{itemize}
\item  The {\sl cluster set}  of $h$ at $\sigma$ is
$$
\Gamma(h, \sigma):=\{p\in\C_\infty: \exists \{z_n\}\subset \D, \lim_{n\to\infty}z_n=\sigma, \lim_{n\to\infty}h(z_n)=p\}.
$$
\item The {\sl non-tangential cluster set}  of $f$ at $\sigma$ is
$$
 \Gamma_{N}(h, \sigma) :=\{p\in\C_\infty: \exists \{z_n\}\subset \D, \text{ $\{z_n\}$ converges non-tangentially to $\sigma$,} \lim_{n\to\infty}h(z_n)=p\}.
$$
\item The {\sl $H$-cluster set of radius $R$} of $f$ at $\sigma$ is
$$
\Gamma_H(h, \sigma, R):=\{p\in\C_\infty: \exists \{z_n\}\subset E(\sigma, R), \lim_{n\to\infty}z_n=\sigma, \lim_{n\to\infty}h(z_n)=p\}.
$$
\end{itemize}

Note that $\Gamma(h, \sigma)=\{p\}$ if and only if $h$ has \emph{ (unrestricted) limit} $p\in\C_\infty$ at $\sigma\in\partial \D$, that is $\lim_{\D\ni z\to \sigma}h(z)=p$. Also, $\Gamma_{N}(h, \sigma)=\{p\}$ if and only if $h$  has \emph{non-tangential limit} $p\in\C_\infty$ at $\sigma\in\partial \D$---for short, $\angle\lim_{z\to\sigma}h(z)=p$. That is, $\{h(z_n)\}$ converges at $p$ for every sequence $\{z_n\}\subset\D$ which converges non-tangentially to $\sigma$. 

For Riemann maps there is a strong connection between the first two cluster sets and principal parts/impressions of corresponding prime ends. Namely, if $h$ is univalent and $\Omega=h(\D)$ (see, {\sl e.g.}, \cite[Proposition~4.4.4 and Theorem~4.4.9]{BCDbook})
\[
\Gamma(h,\sigma)=I(\hat{f}(\underline{\sigma})), \quad \Gamma_N(h,\sigma)=\Pi(\hat{h}(\underline{\sigma})).
\]

The next lemma shows that $H$-cluster sets are actually independent of $R$. 
\begin{lemma}\label{lemma:H-lim-univ}
Let $h:\D\to \C$ be holomorphic. Then, for every $R, R'>0$
\[
\Gamma_H(h, \sigma, R)=\Gamma_H(h, \sigma, R').
\]
In particular,  $h$ has $H$-limit at $\sigma$ if and only if $\Gamma_H(h, \sigma, R)$ is a point for some---and hence any---$R>0$.
\end{lemma}
\begin{proof}
We can assume that $R>R'$. It is clear that $\Gamma_H(h, \sigma, R')\subseteq\Gamma_H(h, \sigma, R)$ since $E(\sigma, R')\subset E(\sigma, R)$. We have thus to show that $\Gamma_H(h \sigma, R)\subseteq\Gamma_H(h, \sigma, R)$. To this aim, it is enough to show that 

\smallskip

{ ($\ast$)}  if  $\{z_n\}\subset E(\sigma, R)$ is such that $\{z_n\}\ra \sigma$ and $\{h(z_n)\}\ra p\in\C_\infty$ then there exists a sequence $\{\tilde z_n\}\subset E(\sigma, R')$  such that $\{\tilde z_n\}\ra\sigma$ and $\{h(\tilde z_n)\}\ra p\in\C_\infty$.

\smallskip

Let $h_\sigma:\D\to \Ha$ be the biholomorphism given by $h_\sigma(z)=\frac{\sigma+z}{\sigma-z}$. An explicit computation shows that for all $M>0$
\[
h_\sigma(E(\sigma, M))=\{w\in\C: \Re w>\frac{1}{M}\}.
\]
Let $g:=h\circ h_\sigma^{-1}:\Ha \to\C$. Hence, in order to prove ($\ast$) it is enough to prove that, given $\{w_n\}$ such that $\Re w_n>\frac{1}{R}$, $\{w_n\}\ra\infty$ and $\{g_n(w_n)\}\ra p$ then there exists $\{\tilde w_n\}$ such that $\Re \tilde w_n>\frac{1}{R'}$, $\{\tilde w_n\}\ra \infty$ and $\{g_n(\tilde w_n)\}\ra p$.


Let $\{w_n\}$ be a sequence as before. If there is a subsequence $\{w_{n_k}\}$ such that $\Re w_{n_k}>\frac{1}{R'}$, then we can take $\tilde w_k:=w_{n_k}$ and we are done. So we can assume that for all $n$,
\begin{equation}\label{Eq:inRR}
\frac{1}{R}<\Re w_n<\frac{1}{R'}
\end{equation}

  For all $n$, set $\tilde w_n:=w_n+\frac{1}{R'}-\frac{1}{R}$. Clearly, $\Re \tilde w_n>\frac{1}{R'}$ and  $\{\tilde w_n\}$ converges to $\infty$. We claim that $\{g(\tilde w_n)\}$ converges to $p$, and thus we are done. 
 
Let $\Omega:=g(\Ha)=f(\D)$. Assume that $\{g(\tilde w_n)\}$ converges, up to subsequences, to a point $q\in \overline{\Omega}$. Thus we have to show that $p=q$.

Since $g$ is holomorphic, it contracts the hyperbolic distance, and by \eqref{Eq:inRR} holds, using \eqref{Eq:Poinc-semip}   we have 
\[
k_\Omega(g(w_n), g(\tilde w_n))\leq k_\Ha(w_n, \tilde w_n)\leq \int_{\re w_n}^{\re \tilde w_n}\frac{1}{2\re t}dt\leq\frac{ R}{2}(\frac{1}{R'}-\frac{1}{R}).
\]

Therefore, the sequences $\{g(w_n)\}$ and $\{g(\tilde w_n)\}$ stay at finite hyperbolic distance in $\Omega$. Since $\{g(\tilde w_n)\}$ converges to $p\in\partial\Omega$ and $k_\Omega$ is complete, it follows that also $q\in\partial\Omega$. Therefore, by the so-called ``distance lemma'' (see for example \cite[Theorem~5.3.1]{BCDbook}) $p=q$.
\end{proof}

The previous lemma allows as to define the {\sl $H$-cluster set} of a holomorphic map $h:\D\to \C$ at $\sigma\in\partial \D$ as
\[
\Gamma_H(h, \sigma):=\Gamma_H(h, \sigma, R),
\]
where $R$ is any real positive number. 

Note that for all $R>0$ and $\sigma\in\partial\D$, the horocycle $E(\sigma, R)$ is an Euclidean disc with center $\frac{\sigma}{1+R}$ and radius $R/(1+R)$, contained in $\D$ and tangent to $\partial\D$ at $\sigma$. Hence, 
\[
\Gamma_N(f, \sigma)\subset \Gamma_H(f, \sigma)\subset \Gamma(f, \sigma).
\]
 Propositions~\ref{Thm:starlike}  and ~\ref{prop:Carsten}   imply that in general the previous sets are different.
 
Let now $A_R$ be the affine transformation which maps $\D$ onto $E(\sigma,R)$ such that $A_R(\sigma)=\sigma$. Therefore, if $h:\D\to \C$ is univalent and we let $h_R:\D\to \C$ be defined as $h_R:=h\circ A_R$, the previous discussion implies that for all $R>0$
\[
I(\hat{h_R}(\underline{\sigma}))=\Gamma_H(h, \sigma).
\]

More interesting for our aims is the following:

\begin{proposition}\label{Prop:simil-equal}
Let $h:\D\to\C$ be univalent and let $\Omega:=h(\D)$. Let $\sigma\in \partial\D$ and let $\gamma:[0,+\infty)\to\Omega$ be a geodesic ray such that $\gamma(t)$ converges to the prime end $\hat{h}(\underline{\sigma})$ in the Carath\'eodory topology of $\Omega$. Then for all $R>0$
\[
\overline{E_\Omega(\gamma, R)}\cap \partial\Omega=\Gamma_H(h, \sigma).
\]
\end{proposition}
\begin{proof}
Let $\eta:[0,1)\to \Omega$ be defined by $\eta(r):=h(r\sigma)$. The curve $\eta$ is a geodesic ray of $\Omega$ which converges to the prime end $\hat{h}(\underline{\sigma})$ in the Carath\'eodory topology of $\Omega$. By \eqref{Eq:oro-in-disc-back}, $h(E(\sigma, M))=E_\Omega(\eta, M)$ for all $M>0$. Hence, it is clear that for all $M>0$,
\[
\overline{E_\Omega(\eta, M)}\cap \partial\Omega=\Gamma_H(h, \sigma, M)=\Gamma_H(h, \sigma).
\]
The result follows then from Lemma~\ref{lemma:oro-in-simply-change}.
\end{proof}

\section{General statement and  proof of Theorem~\ref{Thm:main Intro}}\label{sec:proof}
In this section we prove the following  more general version of Theorem ~\ref{Thm:main Intro}

\begin{theorem}\label{Thm:main}
Let $\Omega\subsetneq \C$ be a simply connected domain. The following are equivalent:
\begin{enumerate}
\item $\Omega$ has the Denjoy-Wolff Property; \item every $g\in \hbox{Aut}_d(\Omega)$ has a Denjoy-Wolff point;
\item every parabolic $g\in \hbox{Aut}_d(\Omega)$ has a Denjoy-Wolff point;
\item for every geodesic ray $\gamma$ in $\Omega$  there exists $p\in\partial\Omega$  such that $\overline{E_\Omega(\gamma, M)}\cap \partial \Omega=\{p\}$ for  every $M\in \R$;
\item there exist $z_0\in\Omega$ such that for  every geodesic ray $\gamma$ in $\Omega$ with  $\gamma(0)=z_0$ there exists  $p\in\partial\Omega$ and $M\in \R$ such that $\overline{E_\Omega(\gamma, M)}\cap \partial\Omega=\{p\}$;
\item for every Riemann map $h:\D\to \Omega$, $h$ has $H$-limit at every $\sigma\in\partial\D$;
\item there exists a Riemann map $h:\D\to \Omega$, such that $h$ has $H$-limit at every $\sigma\in\partial\D$.
\end{enumerate}
\end{theorem}

In all this section, $\Omega\subsetneq\C$ is a simply connected domain. First of all, we translate the Denjoy-Wolff theorem in $\Omega$ in terms of the existence of invariant horospheres:

\begin{lemma}[Existence of invariant horospheres]\label{lemma:DW-abstract}
Let $f\in \hbox{Hol}_d(\Omega,\Omega)$. Then there exists a geodesic ray $\gamma$ in $\Omega$ such that for all $R>0$
\[
f(E_\Omega(\gamma, R))\subseteq E_\Omega(\gamma, R).
\]
Moreover, if $\eta$ is another geodesic ray in $\Omega$ such that $f(E_\Omega(\eta, R))\subseteq E_\Omega(\eta, R)$ for all $R>0$ then $\eta$ and $\gamma$ converge in the Carath\'eodory topology of $\Omega$ to the same prime end of $\Omega$. 

In particular, the limit set of $\{f^{\circ n}\}$ is contained in $\overline{E_\Omega(\gamma, R)}\cap \partial\Omega$.
\end{lemma}
\begin{proof}
Let $h:\D\to \Omega$ be a Riemann map. Hence, $\tilde f:=h^{-1}\circ f\circ h\in \hbox{Hol}_d(\D,\D)$. By the Denjoy-Wolff theorem, there exists a unique point $\tau\in\partial\D$ such that $\tilde f(E(\tau,R))\subseteq E(\tau, R)$ for all $R>0$. Let $\gamma:[0,1)\to\Omega$ be the geodesic ray defined by $\gamma(r):=h(r\tau)$. Hence, by \eqref{Eq:oro-in-disc-back}, $f(E_\Omega(\gamma, R))\subseteq E_\Omega(\gamma, R)$ for all $R>0$. 

Note that $\gamma$ converges  in the Carath\'eodory topology of $\Omega$ to the prime end $\hat{h}(\underline{\tau})$. If $\eta$ is any geodesic ray in $\Omega$ such that $f(E_\Omega(\eta, R))\subseteq E_\Omega(\eta, R)$ for all $R>0$, then $\tilde f(E_\D(h^{-1}\circ \eta, R))\subseteq E_\D(h^{-1}\circ \eta, R)$ for all $R>0$. 

Let $\sigma\in\partial\D$ be such that $h^{-1}\circ \eta$ converges to $\sigma$---or, equivalently, $\eta$ converges  in the Carath\'eodory topology of $\Omega$ to the prime end $\hat{h}(\underline{\sigma})$. Let $\eta_\sigma(r):=r\sigma$, $r\in [0,1)$. Hence, by \eqref{eq-limit-oro-eucl}, for all $R>0$, $E_\D(\eta_\sigma, R)=E(\sigma, R)$. Thus, by \eqref{Eq:oro-cambia}, there exists $T>0$ such that $E_\D(h^{-1}\circ \eta, R)=E(\sigma,T+R)$ for all $R>0$. Therefore, $\tilde f(E(\sigma, R))\subseteq E(\sigma, R)$ for all $R>0$ and by the Denjoy-Wolff theorem $\sigma=\tau$. Hence, $\eta$ converges  in the Carath\'eodory topology of $\Omega$ to the prime end $\hat{h}(\underline{\tau})$.
\end{proof}

Now, we are ready to prove Theorem~\ref{Thm:main}.

\begin{proof}[Proof of Theorem~\ref{Thm:main}]
Trivially, (1) implies (2) and (2) implies (3).

Also, (6) and (7) are equivalent since Riemann mappings are unique up to pre-composition with automorphisms of $\D$. While, (4), (5) and (6) are equivalent by Proposition~\ref{Prop:simil-equal}. 
Also, (5) implies (1) via Lemma~\ref{lemma:DW-abstract}.

Hence, it is enough to show that (3) implies (6). 

Let $h:\D\to \Omega$ be a Riemann map. Let $\sigma\in\partial\D$. We have to show that $h$ has $H$-limit at $\sigma$. 

As in the proof of Lemma~\ref{lemma:H-lim-univ}, we can  work in $\H$ by  considering the  biholomorphism $h_\sigma:\D\to\Ha$ defined by $h_\sigma(z)=\frac{\sigma+z}{\sigma-z}$. Let $\tilde h:=h\circ h_\sigma^{-1}:\Ha\to \Omega$. Hence, we have to show that there exists $p\in\partial\Omega$ such that for every fixed $R>0$ and every sequence $\{w_n\}$ with $\Re w_n\geq R$ and $|w_n|\to \infty$ then $\{\tilde h(w_n)\}$ converges to $p$.

Let $T:\Ha\to \Ha$ be the automorphism defined by $T(w)=w+i$. Note that $T$ is a parabolic automorphism of $\Ha$ (that is, $\lim_{n\to\infty} \frac{k_\Ha(g^{\circ n}(w_0), w_0)}{n}=0$ for all $w_0\in \Ha$). Hence, $g:= \tilde h  \circ T \circ \tilde h^{-1}$ is a parabolic automorphism of $\Omega$. 

By hypothesis (3), there exists $p\in\partial\Omega$ such that $\{g^{\circ n}\}$ converges uniformly on compacta to $p$.  Let $z_0\in \R\cap\H$. Hence, $\{g^{\circ n}(\tilde h(z_0))\}$ converges to $p$. That is, 
\[
\lim_{n\to \infty}\tilde h(z_0+in)=p.
\]
We claim that 
\begin{equation}\label{Eq:lim1}
\lim_{t\to +\infty}\tilde h(z_0+it)=p.
\end{equation}
 Indeed, since the hyperbolic discs in $\Ha$ are Euclidean discs ($\Ha$ is convex),  it follows immediately that for every $z,w,u\in \Ha$ and $s\in[0,1]$
\[
k_\Ha(z, sw+(1-s)u)\leq \max \{k_\Ha(z, w), k_\Ha(z, u)\}.
\]
Hence, denoting by $[t]$ the integer part of a real number $t$, 
\begin{equation*}
\begin{split}
k_\Omega(\tilde h(R+it), \tilde h (R+i[t]))&=k_\Ha(R+it, R+i[t])\\& \leq k_\Ha(R+i[t+1], R+i[t])=k_\Ha(R+i, R),
\end{split}
\end{equation*}
where, for the last equality we used that $w\mapsto w+i[t]$ is an automorphism of $\Ha$ for $t\geq 0$ fixed.

Since $\tilde h$ is proper and $R+it\to \infty$ for $t\to+\infty$, by the distance lemma (see, {\sl e.g.}, \cite[Theorem~5.3.1]{BCDbook}), $\lim_{t\to+\infty} \tilde h(R+it)=p$. 

By Lehto-Virtanen's theorem (see, {\sl e.g.} \cite[Theorem~3.3.1]{BCDbook}), $\tilde h $ has non-tangential limit $p$ at $\infty$. In particular, 
\begin{equation}\label{Eq:lim2}
\lim_{r\to+\infty}\tilde h(R+r)=p.
\end{equation}
Let $\Gamma$ be the Jordan curve given by $\overline{(R+i[0,+\infty))}\cup \overline{(R+[0,+\infty))}$ and let $U$ be the connected domain bounded by $\Gamma$ which contains $R+1+i$. Note that $\overline{U}\setminus\{\infty\}\subset\Ha$. 

By \eqref{Eq:lim1} and \eqref{Eq:lim2}, $\tilde h(\Gamma)$ is a Jordan curve in $\C_\infty$ and it is clear that one of the connected component of $\C_\infty\setminus \tilde h(\Gamma)$ is $\tilde h(U)$. Therefore, $\tilde h|_U:U\to \tilde h (U)$ is a biholomorphism and, since $U$ and $\tilde h (U)$ are Jordan domain, by Carath\'eodory's extension theorem, $\tilde h|_U$ extends as a homeomorphism $\overline{U} \to \overline{\tilde h(U)}$ that we still denote by $\tilde h$. Since $\tilde{h}^{-1}(p)=\infty$, it follows that $\lim_{n\to +\infty} \tilde h(w_n)=p$ for every $\{w_n\}\subset U$ (that is, $\Re w_n\geq R$ and $\Im w_n\geq 0$)  converging to $\infty$. 

In order to end the proof, we repeat the same construction replacing $T$ with $T^{-1}$ (that is, the parabolic automorphism $w\mapsto w-i$) and we show that there exists $q\in\partial\Omega$ such that $\lim_{n\to +\infty} \tilde h (w_n)=q$ for all $\{w_n\}$ converging to $\infty$ such that $\Re w_n\geq R$ and $\Im w_n\leq 0$. But then, taking $w_n=n$, we see that $p=q$ and we are done.
\end{proof}

\section{Non-equivalence of limits: Construction of examples}

\subsection{H-limits are weaker than unrestricted limits: proof of Propositon~\ref{Thm:starlike}}

In this section we construct an example of a simply connected domain which is not locally connected, yet H-limits exist for any Riemann map (See Figure~\ref{fig1}). This shows that the two conditions are indeed non equivalent. 

 As a notation,  for $a\in \R$, we consider the horizontal line
\[
L_{a}:=\{z\in \C: \Re z \leq 0, \Im z=a\}.
\]
Let $\mathcal A:=\{a_n\}$ be a decreasing sequence of positive real numbers converging to $0$. Let
\begin{equation}\label{Eq:D-doman}
D:=\{z\in \C: -1<\Re z<1, -1<\Im z<1\}\setminus \left(L_0\cup \bigcup_{n\in \N}L_{\pm a_n}\right).
\end{equation}

\begin{figure}[hbt!]
\centering
\setlength{\unitlength}{0.9\textwidth}
\includegraphics[width=0.9\textwidth]{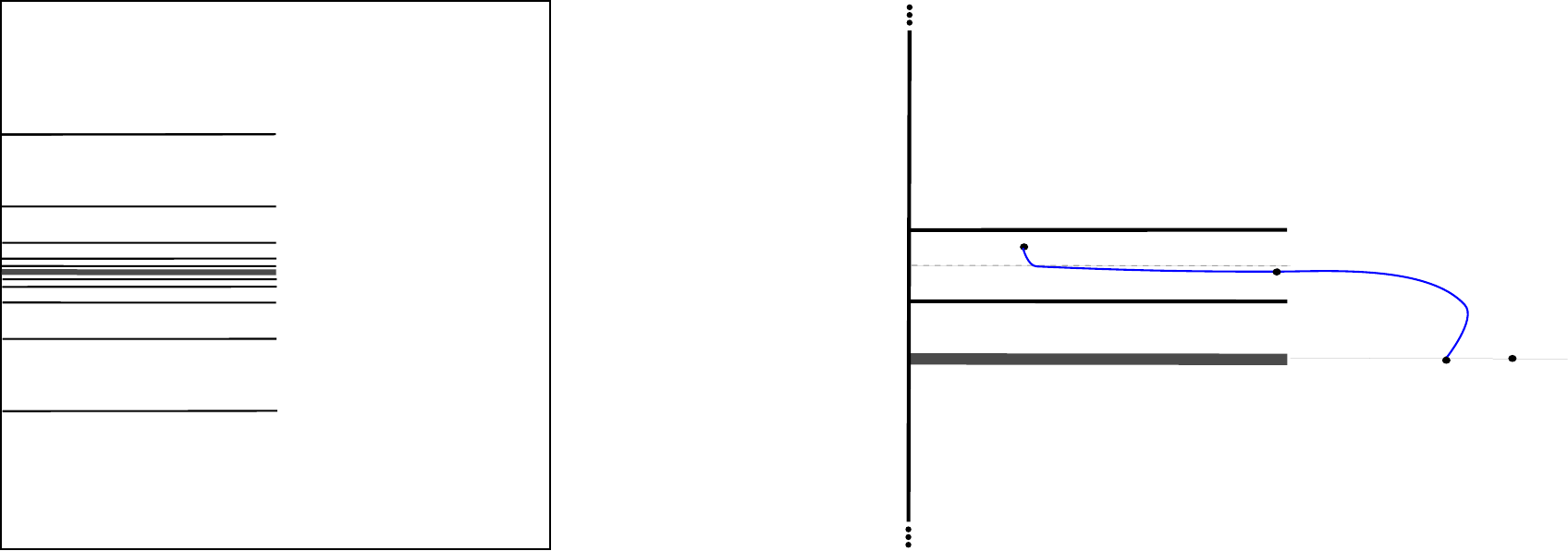}
\put(-0.2,0.19){\small $q_n$}
\put(-0.38,0.19){\small $z_n$}
\put(-0.37,0.09){\small $-k$}
\put(-0.21,0.09){\small $-h$}
\put(-0.09,0.09){\small $r$}
\put(-0.06,0.09){\small $r_0$}
\caption{\small \label{fig1} An illustration of the domain $D$ and of the proof of Proposition~\ref{Prop:good-domain}}
\end{figure}
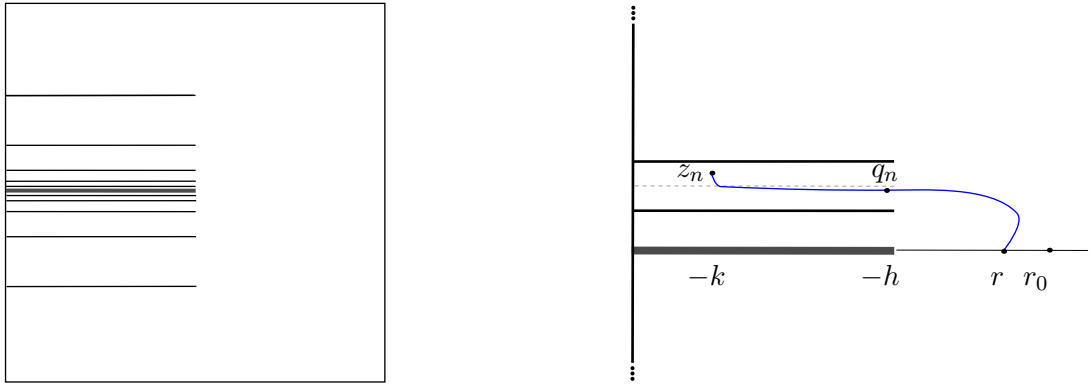

For $n\geq 1$ let $S_n\subset D$ be the rectangles 
\begin{equation}\label{eq:rectangles}
S_n:=\{ -1<\re z<0,\ a_{n+1}<\Im z<a_{n} \},
\end{equation}
and let $\epsilon_n:=(a_{n}-a_{n+1})/2$ be the halfwidth of these rectangles. Let $S^\ast_n$ be the reflection of $S_n$ with respect to the real axis; then $S^\ast_n$ has the same half width as $S_n$.
 For $n$ which tends to infinity, $\epsilon_n$ tends to $0$.

Note that $D$ is simply connected and bounded and not locally connected at any point $-k\in\R_-\cap D$. Moreover, using Carath\'eodory's prime end theory, it is easy to see that the principal part of every prime end of $D$ is a point. Also, there is a prime end $\underline{\xi_0}\in\partial_C D$, represented by the null-chain $(C_n)$ with 
\begin{equation}\label{Eq:special-null}
C_n:=\{z\in\C: |z|=a_n, \Re z\geq 0\}
\end{equation}
whose principal part is $\{0\}$ and whose impression is $\{z\in \C: \Im z\in [0,1], \Re z=0\}$. While, for all $\underline{\xi}\in \partial_C D\setminus \underline{\xi_0}$ the impression of $\sigma$ is a point. 

We will show:

\begin{proposition}\label{Prop:good-domain}
Let $D\subset \C$ be the domain defined in \eqref{Eq:D-doman}, $r_0\in(0,1)$. Then   for  every geodesic ray $\gamma$ in $D$ such that $\gamma(0)=r_0$ there exists  $p_\gamma\in\partial D$ such that $\overline{E_D(\gamma, 0)}\cap \partial D=\{p_\gamma\}$.\end{proposition}

Hence, the Proposition~\ref{Prop:good-domain} implies that $\Omega$ satisfies condition (4) of Theorem~\ref{Thm:main}. Therefore, Proposition~\ref{Thm:starlike}  follows.
In order to prove Proposition~\ref{Prop:good-domain} we need some technical lemmas.

\begin{lemma}\label{Lem:stima-up}
Suppose that $a, a'\in \R$ and $a'> a$.  Let $\Omega\subsetneq \C$ be a simply connected domain such that $L_{a}, L_{a'}\subseteq \C\setminus\Omega$. Let $h>0$ and set $\epsilon:=\frac{a'-a}{2}$. 

Let  $q\in\Omega$ with $\Re q=-h$ and set $\theta:=\Im q-\epsilon-a$. Clearly $\theta\in(-\epsilon,+\epsilon)$. Suppose that $q+t\in \Omega$ for all $t\geq 0$.  Let $r_0>a$. Then
\[
k_\Omega(q, r_0+i\;\Im q)\leq \frac{h}{\epsilon-|\theta|}+(\sinh)^{-1}\left(\frac{r_0}{\epsilon-|\theta|} \right).
\]
\end{lemma}

\begin{proof} Let $y_0:=\theta+\epsilon+a$. Let $\eta(t):=t+iy_0$ for $t\in [-h, r_0]$. By Lemma~\ref{lem:hyp metric on simply connected}, for all $t \in [-h, r_0]$ we have
\begin{equation*}
\kappa_{\Omega}(\eta(t);\eta'(s))\leq \frac{|\eta'(t)|}{\dist(\partial\Omega,\eta(t)).}
\end{equation*}
Therefore,
\begin{equation}\label{Eq:stima-above}
k_\Omega(q, r_0+iy_0)\leq \int_{-h}^{r_0} \kappa_\Omega(\eta(t); \eta'(t))dt\leq \int_{-h}^{r_0} \frac{dt}{\dist(\partial\Omega,\eta(t))}.
\end{equation}
Now,
\[
\dist(\Omega,\eta(t))=\begin{cases}  
\epsilon-|\theta| & t\in [-h, 0]\\
\sqrt{(t^2+(\epsilon-|\theta|)^2} &  t\in [0, r_0]
\end{cases}
\]
Hence, by \eqref{Eq:stima-above}, we have
\[
k_\Omega(q, r_0+iy_0)\leq \int_{-h}^{0} \frac{dt}{\epsilon-|\theta|}+\int_{0}^{r_0} \frac{dt}{\sqrt{(t^2+(\epsilon-|\theta|)^2}}=\frac{h}{\epsilon-|\theta|}+(\sinh)^{-1}\left(\frac{r_0}{\epsilon-|\theta|} \right),
\]
and we are done.
\end{proof}

Therefore, given any  $c,C$ with $0<c<C<1$,  we have that for $n$ sufficiently big, $4C\epsilon_n <1$ and hence  for all $u\in (0,1)$,
\[
B_n:=\{z\in D: a_n<\Im z<a_{n-1} : -u+c\epsilon_n <\Re z<-C\epsilon_n \}\subset D.
\]
The open set $B_n$ is a ``good box'' in the sense of \cite[Definition~5.7]{BCDG} (note that in \cite{BCDG}  one considered vertical lines, while here we consider horizontal lines, so that the real and the imaginary parts have  to be switched and the ``width'' in that definition corresponds to $2\epsilon_n$ in our case). Similarly, $B_n^\ast$ (the reflection of $B_n$ through the real axis) are ``good boxes''. In \cite{BCDG} it has been proved that in a good box the hyperbolic geometry is essentially that of a strip, in particular, geodesics joining ``sufficiently far away'' points  tends to stay in the middle of the box. More precisely, as a consequence of \cite[Corollary~5.8]{BCDG}, we have:

\begin{lemma}\label{Lem:stima-down}
Let $D$ be the domain defined in \eqref{Eq:D-doman}. Let $k\in (0,1)$,  $C_0\in (0,1)$ and $h\in (0,k)$. Then there exists  $\epsilon>0$ such that whenever the strip $S_n$ has half width $\epsilon_n<\epsilon$, the following holds.  For  every $0<r<r_0$ and $z\in S_n$ such that $\Re z=-k$, $\Im z>0$, if $\gamma$ is the geodesic connecting $z$ and $r$,  then there exists $q_n=q_n(h,r)\in\gamma$ such that $\Re q_n=-h$ and $|\Im q_n-(\epsilon_n+a_n)|\leq C_0\epsilon_n$.
\end{lemma}

Now we are in good shape to give

\begin{proof}[Proof of Proposition~\ref{Prop:good-domain}]
Let $r_0\in(0,1)$.
Since the domain $D$ is symmetric with respect to the real axis, in view of Remark~\ref{rem:symmetric domains}, we can choose a Riemann map $f:\D\to D$ such that  $f(0)=r_0$ and $\ov{f(\ov{z})}=f(z)$, so that $D\cap\R$ is a geodesic.  We can  also assume that $f(x)\ra0^+$ as $x\ra1^-$.

 For $\sigma\in\partial \D$ let  $R_\sigma$ 
  be the geodesic ray in $\D$ whose closure connects $0$ to $\sigma$ and let $\gamma_\sigma:=f(R_\sigma)$. Since $f(0)=r_0$, for every geodesic ray  $\gamma$ in $D$ originating at  $r_0$, $\gamma=\gamma_\sigma$ for some $\sigma\in\partial\D$. 
Because of the choice that $f(x)\ra0$ as $x\ra1$, the map $f$ extends continuously to any $\sigma\in\partial \D$ with $\sigma\neq1$, so $\hat{f}(1)$  is the only prime end of $D$ whose impression is not a point.  In particular, by \eqref{Eq:oro-in-disc-back} $\overline{E_D(\gamma_\sigma, M)}\cap \partial D=f(\sigma)$ for all $\sigma\neq1$. So we only need to study  Busemann horospheres centered at $\gamma_1=(0,r_0)$.

We prove the claim by showing that for every $M\in\R$, 
$$
\overline{E_D(\gamma_1, M)}\cap \partial D=\{0\}.
$$
In other words, since $0$ is an accumulation point of $\gamma_1$ in $\partial D$ we need to show that $\overline{E_D(\gamma, 0)}\cap \partial D=\{0\}$.  Observe that 
 \begin{equation}\label{eq:limit of horosphere in negative segment}
\overline{E_D(\gamma_1, M)}\cap \partial D\subseteq [-1,0].
\end{equation}
Indeed let  $\Gamma(f,1)$ be the cluster set of $f$ at $1$.   Clearly, 
$\overline{E_D(\gamma_1, M)}\cap \partial D\subseteq \Gamma(f,1)$. As we already recalled, 
 $\Gamma(f,1)$ coincides with the impression of the prime end $\hat{f}(\underline{1})$, hence \eqref{eq:limit of horosphere in negative segment} follows. 
 
Assume by contradiction that  there exists $M\in\R$  such that $[-1,0)\cap \overline{E_D(\gamma_1, M)}$ contains at least two points. Then we claim that there exist $k\in (0,1)$, an  infinite set $\NN\subseteq \N$ and sequence $\{z_n\}_{n\in\NN}\subset E_D(\gamma_1, M)$ such that  (recall the definition of the rectangles $S_n$ in \eqref{eq:rectangles})
\begin{equation}\label{eq:replacing zn}
\begin{cases}
\re z_n=-k, & \ z_n\in S_n  \quad \text{ for all $n\in\NN$},\\
\Im z_n\searrow0^+,
\end{cases}
\end{equation}
 and for all $n\in\NN$,
\begin{equation}\label{Eq:limit-oro-fails}
\lim_{r\to 0^+}[k_D(z_n, r)-k_D(r_0, r)]< M.
\end{equation}

Indeed, if $\overline{E_D(\gamma, 0)}\cap \partial D$ is not reduced to $\{0\}$, by \eqref{eq:limit of horosphere in negative segment} there exists $s\in (0,1]$ such that $-s\in \overline{E_D(\gamma, 0)}\cap \partial D$. Thus there exists a sequence $\{w_n\}$ such that $w_n\in E_D(\gamma_1, M)$ for all $n$ and $\lim_{n\to\infty}w_n=-s$. Since $E_D(\gamma_1, M)$ is symmetric with respect to the real axis (see Remark~\ref{rem:symmetric domains}), it follows that also $\{\overline{w_n}\}\subset E_D(\gamma_1, M)$ and $\lim_{n\to\infty}\overline{w_n}=-s$. Therefore, there exists a sequence, which we still denote by $\{w_n\}$, such that $\Im w_n>0$, $w_n\in E_D(\gamma_1, M)$ for all $n$ and $\lim_{n\to\infty}w_n=-s$. Up to passing to a subsequence, we can also assume that $\{\Im w_n\}$ is strictly decreasing and for every $n$ there exists $m_n\in\N$ such that $w_n\in S_{m_n}$, with $S_{m_n}\cap S_{m_{\tilde n}}=\emptyset$ for $n\neq \tilde n$. Note in particular that $\{m_n\}$ is strictly increasing and converges to $\infty$.

Now, fix $k\in (0,s)$. For all $n$ sufficiently large, since $w_n\ra -s$, we have that $\re w_n< -k$. Thus, again up to passing to a subsequence, we can assume that $\re w_n< -k$ for all $n$. Let $\alpha_n$ be the geodesic connecting $w_n$ to $r_0$. By Remark~\ref{rem:sliding along geodesics},  $\alpha_n\subset E_D(\gamma_1, M)$ for all $n$.  Since $\alpha_n$ is a continuous curve and $-s<-k<0$, by the intermediate value theorem $\alpha_n$ contains points of real part equal to $-k$. By letting $z_n$ to be the first of such points when sliding from $w_n$ to $r_0$ along $\alpha_n$, we have that $z_n\in S_{m_n}$ (because to exit the rectangle $S_n$, the curve $\alpha_n$ first needs to touch the set $\Re z=0$). Hence, by construction, $\{z_n\}\subset E_D(\gamma_1, M)$ is a sequence such that $\Re z_n=-k$, and $z_n\in S_{m_n}$ for all $n$. Since $\{m_n\}$ is strictly increasing and converges to $\infty$, this implies that  $\{\Im z_n\}$ strictly decreases to $0$. Setting $\NN=\{m_n\}$ and relabelling the indices of $\{z_n\}$ by letting $z_{m_n}=z_n$ with $z_n\in S_{n_m}$, we have thus proved the claim (\eqref{Eq:limit-oro-fails} just means that $\{z_n\}\subset E_D(\gamma_1, M)$).

Thus, arguing by contradiction, we can assume that both \eqref{eq:replacing zn} and \eqref{Eq:limit-oro-fails} hold. 

Let $r\in (0,1)$, $h\in(0,\frac{k}{17})$ and $C_0=\frac{1}{2}$. For each  $n\in\NN$ sufficiently large so that $\epsilon_n$ is sufficiently small let  $q_n$ be a point as given by Lemma~\ref{Lem:stima-down}, which belongs to the geodesic joining $z_n$ and $r$, and  such that  
\begin{equation}\label{eq:stima-im-qn}
\Re q_n=-h \text{\ \ \ and\ \ \ }|\Im q_n-(\epsilon_n+a_n)|\leq \frac{\epsilon_n}{2}.
\end{equation}

We first estimate $k_D(z_n, q_n)$. 
For each $n\in\NN$ consider the strip $\SS_n:=\{a_n<\Im z<a_{n-1}\}$ and the domain  $D_n:=\C\setminus (L_{a_n}\cup L_{a_{n-1}})$.  
Notice that $D_n\supset \SS_n$ and that the points $z_n, q_n$ belong to $\SS_n$. Let  $C_1=\frac{1}{2}$. By  localization of hyperbolic distances (see  \cite[Proposition~6.8.3]{BCDbook}),  for  every strip $\SS_n$ with halfwidth  $\epsilon_n$  sufficiently small we have 
\[
k_{D_n}(z_n, q_n) \geq \frac{1}{2} k_{\mathcal S_n}(z_n, q_n).
\]
On the other hand, using Lemma~\ref{lem:hyp metric on simply connected}, since $\dist(z,\partial\SS_n)<\epsilon_n$ for all $z\in\SS_n$ and $|\Re z_n-\Re q_n|=k-h$, we have  
$$
k_{\mathcal S_n}(z_n, q_n)\geq\int_{-h}^{-k}\frac{1}{4\epsilon_n}=\frac{k-h}{4\epsilon_n}.
$$
By additivity of the hyperbolic metric  along geodesics, and since $D\subset D_n$ and $C_1=\frac{1}{2}$, it follows that 
\begin{equation}\label{eq: lower bound zn r}
\begin{split}
k_D(z_n, r)&=k_D(z_n, q_n)+k_D(q_n, r) \geq k_{D_n}(z_n, q_n)+k_D(q_n, r) \\& \geq\frac{(k-h)}{8\epsilon_n}+ k_D(q_n,r).
\end{split}
\end{equation}

We now show that equation (\ref{eq: lower bound zn r}) implies a lower bound on $k_D(q_n, r_0)$. Indeed, let  $n$ be sufficiently large such that the previous estimates hold. By the contradiction assumption, for any $n$ and  for any $r$ sufficiently close to $0$ (depending on $n$) we have
\begin{equation*}
k_D(z _n,r)-k_D(r_0, r)\leq 2M.
\end{equation*}
Thus, $k_D(z _n,r)\leq k_D(r_0, r)+2M$ and, by (\ref{eq: lower bound zn r}),
\[  \frac{(k-h)}{8\epsilon_n}+ k_D(q_n,r) \leq k_D(r_0, r)+2M.
\]
Therefore, by the triangular inequality,
\begin{equation}\label{eq:first-1}
\frac{(k-h)}{8\epsilon_n}-2M \leq k_D(r_0, r)-k_D( q_n,r)\leq k_D(q_n, r_0).
\end{equation}
We now show that, again for $n$ large, $k_D(q_n, r_0)$ needs to satisfy an upper bound incompatible with (\ref{eq:first-1}), thus reaching a contradiction.

Let $\theta_n:=\Im q_n-\epsilon_n-a_{n}$.
By \eqref{eq:stima-im-qn},  we have that $|\theta_n|\leq \frac{\epsilon_n}{2} $, by which it follows that 
\begin{equation}\label{eq:theta}
\frac{1}{\epsilon_n-|\theta_n|}\leq \frac{2}{\epsilon_n}.
\end{equation}
By the triangular inequality, Lemma~\ref{Lem:stima-up}, (\ref{eq:theta}), and since $(\sinh)^{-1}$ is a monotone function on real numbers  we have  that 
\begin{equation*}
\begin{split}
k_D( q_n,r_0)&\leq k_D(q_n,r_0+i\;\Im q_n)+k_D(r_0+i\;\Im q_n, r_0)\\
&\leq k_D(q_n,r_0+i\;\Im q_n)+ \frac{h}{\epsilon_n-|\theta_n|}+(\sinh)^{-1}\left(\frac{r_0}{\epsilon_n-|\theta_n|} \right)\\
&\leq 
k_D(r_0+i\;\Im q_n, r_0)+\frac{2h}{\epsilon_n}+(\sinh)^{-1}\left(\frac{2r_0}{\epsilon_n} \right).
\end{split}
\end{equation*}
Thus, combining this result with  \eqref{eq:first-1}, we have for all $n\in\NN$ sufficiently large
\[
\frac{(k-h)}{8\epsilon_n}\leq \frac{2h}{\epsilon_n}+ (\sinh)^{-1}\left(\frac{2r_0}{\epsilon_n} \right)+k_D(r_0+i\;\Im q_n, r_0)+2M.
\]
Multiplying by $\epsilon_n$ the previous inequality and taking the limit for $n\to \infty$ (which implies $\Im q_n\to 0$ and $\epsilon_n\to 0$) and taking into account that $\lim_{x\ra0}x\sinh(\frac{c}{x})=0$ for any constant $c$,  we get
\[
\frac{k-h}{8}\leq 2h,
\]
which gives a contradiction since we chose $h\in(0,\frac{k}{17})$.
\end{proof}

\subsection{Nontangential limits are weaker than H-limits}\label{sec:Carsten example}

The aim of this section is to show that existence of non-tangential limits does not imply in general existence of $H$-limits. The construction we present here is based on an idea of Petersen~\cite{PetersenPrimeEnds}, who kindly shared his preprint with the authors. It gives an example of a simply connected domain $D\subsetneq \C$ for which any Riemann map $h:\D\ra D$ has nontangential  limit at every point of $\partial\D$, but fails to have $H$-limit at a boundary point (See Figure~\ref{fig2}). 

For $n=0,1,2,\dots$, we let $x_n:=\frac{1}{2^n}$, $y_n:=\frac{1}{2^n e^{3^n}}$. Consider the simply connected domain
\begin{equation}\label{eq:Carsten domain}
D:=\H\setminus\bigcup_{n\in \N} \{z:\Re z=x_n, |\Im z|\geq y_n\}.
\end{equation}

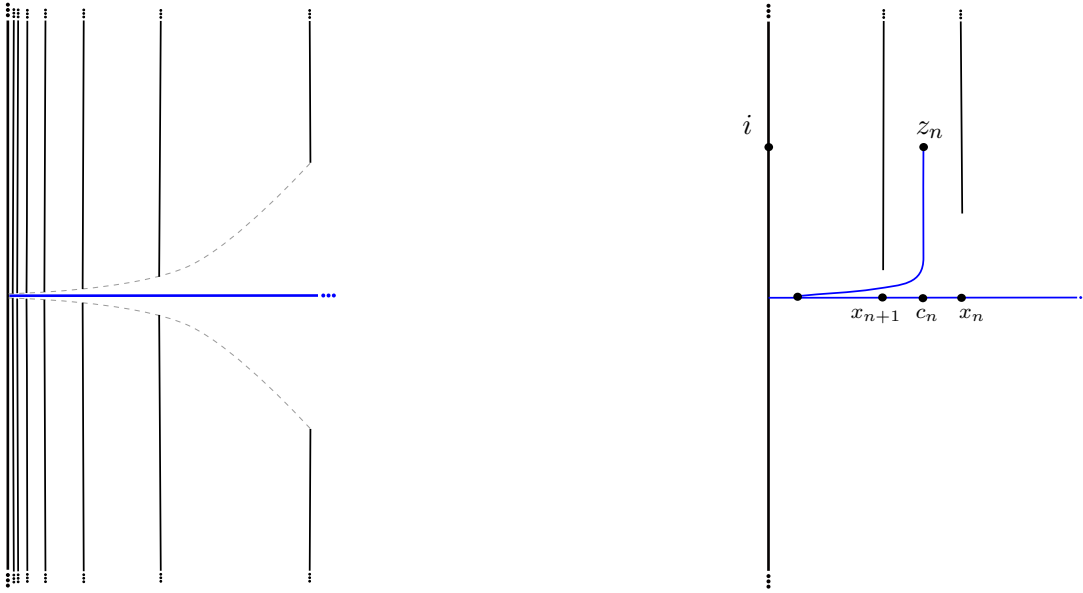
\begin{figure}[hbt!] 
\centering
\setlength{\unitlength}{0.9\textwidth}
\includegraphics[width=0.9\textwidth]{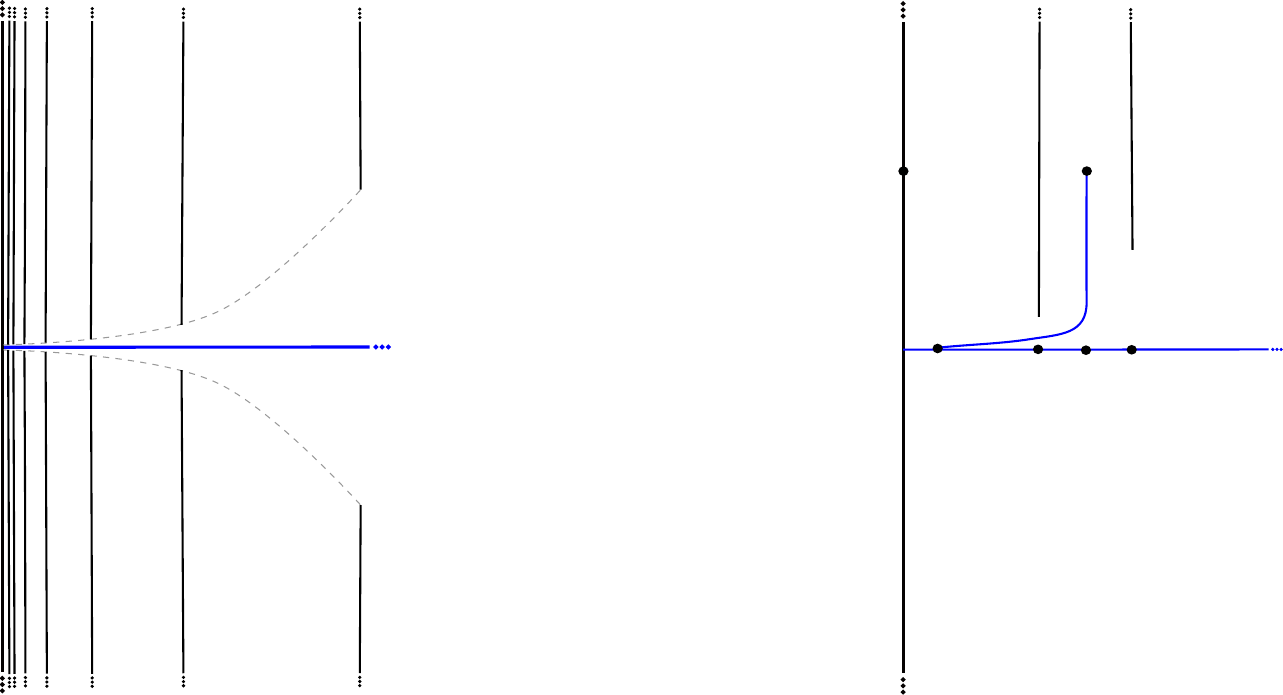}
\put(-0.22,0.25){\tiny $x_{n+1}$}
\put(-0.12,0.25){\tiny $x_{n}$}
\put(-0.16,0.25){\tiny $c_{n}$}
\put(-0.16,0.42){\small $z_n$}
\put(-0.32,0.42){\small $i$}

\caption{\small \label{fig2} An illustration of the domain $D$ and of the proof of (\ref{Eq:contact-oro-Carst}). Geodesics are in blue.}
\end{figure}

By construction, $D$ is symmetric with respect to the real axis. Hence, by Remark~\ref{rem:symmetric domains},  $\gamma_1:=(0,1]$ is (the image of) a geodesic ray in $D$. Also, it is simple to see that the principal part of every prime end of $D$ is a singleton. Moreover, there exists a unique prime end $\underline{\zeta}_0$ of $D$ such that $I(\underline{\zeta})$ is a singleton for all $\underline{\zeta}\in \partial_C D\setminus\{\underline{\zeta_0}\}$ and $I(\underline{\zeta_0})=i\R$---the prime end $\underline{\zeta_0}\}$ is represented by the null chain $(C_n)$ where 
\[
C_n=\{z\in \C: \Re z=x_n, |\Im z|\leq y_n\}.
\]

Hence, if $g:\D\to D$ is a Riemann map, $\sigma_0\in\partial\D$ is such that $\hat{g}(\underline{\sigma_0})=\underline{\zeta_0}$, it follows that (see Subsection~\ref{sec:limits})  $\lim_{\D\ni z\to \sigma}g(z)$ exists for all $\sigma\in\partial\D\setminus\{\sigma_0\}$, $\angle\lim_{\D\ni z\to \sigma_0}g(z)$  exists but $\lim_{\D\ni z\to \sigma_0}g(z)$ does not. We will show that for all $M\in\R$
\begin{equation}\label{Eq:contact-oro-Carst}
\overline{E_D(\gamma_1,M)}\cap\partial D=i\R.
\end{equation} 
Hence, by Theorem~\ref{Thm:main} we have
\begin{proposition}\label{prop:Carsten}
Let $D$ be as in \eqref{eq:Carsten domain}. Let $g:\D\to D$ be a Riemann map. Then $g$ has non-tangential limit at every point of $\partial\D$ and there exists $\sigma_0\in\partial\D$ such that $g$ has unrestricted limit at every point of $\partial\D\setminus\{\sigma_0\}$ and $g$ does not have $H$-limit at $\sigma_0$. Also, the domain $D$ fails to have the Denjoy-Wolff Property.
\end{proposition}

\begin{proof}[Proof of \eqref{Eq:contact-oro-Carst}]
We can parametrize the geodesic ray $\gamma_1$ as $\gamma_1:[0,+\infty)\to\D$ with $\gamma_1(t)=\frac{1}{t+1}$. 

Let $M\in\R$. Let , $\epsilon_n:=\frac{1}{2^{n+2}}$ and $c_n:=\frac{3}{2^{n+2}}$, $n\in\N$. Let $y\in \R$. Let $z_n:=c_n+yi$. Clearly, $\{z_n\}\in D$ and $\{z_n\}$ converges to $yi$. We show that $\{z_n\}$ is eventually contained in $E_D(\gamma_1,M)$, so that $iy\in \overline{E_D(\gamma_1,M)}\cap\partial D$, and, by the arbitrariness of $y$, the result follows.

 Notice that $c_n$ are the midpoints between $x_n,x_{n+1}$ and that   $|x_n-c_n|=|x_{n+1}-c_n|=\epsilon_n$.

In order to  show that $\{z_n\}$ is eventually contained in $E_D(\gamma_1,M)$, we prove that there exists $n_M\in\N$ such that for all $n\geq n_M$,
\begin{equation}\label{Eq:stay-in-oro-def-ex}
\lim_{t\ra +\infty}[k_D(z_n, \gamma_1(t)-k_D(1, \gamma_1(t) )]<M.
\end{equation}

We first estimate $k_D(x_j,c_j)$ and $k_D(c_j, x_{j+1})$ for any $j\geq 0$ (see Figure~\ref{fig2} for the case $y=1$). 
For $t\in [0,\epsilon_n]$ and $z(t)=c_j+t$ we have  
$$
\dist(z(t),\partial D)\leq |z(t)-( x_j+i y_j)|=\sqrt{t^2+y_j^2}.$$

Hence using the lower estimate for the hyperbolic metric on simply connected sets (see Lemma~\ref{lem:hyp metric on simply connected}) and monotonicity of the real logarithmic function we have
$$
k_D(x_j,c_j)\geq\frac{1}{4}\int_{0}^{\epsilon_j}\frac{1}{ \sqrt{t^2+y_j^2}}=\frac{1}{4}\ln\left|\frac{\epsilon_j}{y_j}+\sqrt{1+\frac{\epsilon_j^2}{y_j^2}}\right|\geq \frac{1}{4}\ln\left|\frac{2\epsilon_j}{y_j}\right|=\frac{1}{2}(3^j-\ln 2).
$$
 Similarly, by considering  $z(t)=x_{j+1}-t$  for $t\in [0,\epsilon_n]$ and taking into account that
 $$
 \dist(z(t),\partial D)\leq|z(t)- (x_{j+1}+ iy_{j+1})|=\sqrt{t^2+y_{j+1}^2}
 $$
  we  obtain
$$
k_D(c_j, x_{j+1})\geq \frac{1}{4}\ln\left|\frac{2\epsilon_j}{y_{j+1}}\right|=\frac{3^{j+1}}{4}.
$$
It follows that for any $n\geq 1$ and any  $t$  so that $\gamma_1(t)<c_n$, 
\begin{equation}\label{eq:Carsten1}
\begin{split}
k_D(1, \gamma_1(t))&=\sum_{j=0}^{n-1}\left(k_D(x_j, c_j)+k_D(c_j, x_{j+1})\right)+k_D(x_n, c_n)+k_D(c_n,\gamma_1(t))\\&
\geq \sum_{j=0}^{n-1}\left( \frac{1}{2}(3^j-\ln 2)+\frac{3^{j+1}}{4}\right)+\frac{1}{2}(3^n-\ln 2)+k_D(c_n,\gamma_1(t)) \\
&=\frac{9}{8}3^{n}-\frac{(n+1)\ln2}{2}-\frac{5}{8}+k_D(c_n,\gamma_1(t)).
\end{split}
\end{equation}

We now estimate $k_D(z_n,c_n)$ for any fixed $n$. We consider the curve $\eta_n:[0,1]\to D$ given by $\eta_n(t)=c_n+tyi$. Note that $\dist(\eta_n(t), \partial D)\geq \epsilon_n$ for all $t\in[0,1]$. Hence,  using  the upper estimate for the hyperbolic metric (see Lemma~\ref{lem:hyp metric on simply connected}), we have 
 $$
 k_D(z_n,c_n)\leq \ell_D(\eta_n)=\int_0^1\kappa_D(\eta_n(t); \eta_n'(t))dt\leq  \int_0^1\frac{|\eta'_n(t)|}{\dist(\eta_n(t), \partial D)}dt \leq \frac{|y|}{\epsilon_n}=2^{n+2}|y|.
 $$
 
 Hence for any $n$ and any  $t$  so that $\gamma_1(t)<c_n$
 
\begin{equation}\label{eq:Carsten2}
 k_D(z_n,\gamma_1(t)) \leq k_D(z_n,c_n)+k_D(c_n,\gamma_1(t))\leq  2^{n+2}|y|+k_D(c_n,\gamma_1(t)).
\end{equation}

 Putting together (\ref{eq:Carsten1}) and (\ref{eq:Carsten2}) we obtain for  any  $t$  so that $\gamma_1(t)<c_n$,
\begin{equation*}
\begin{split}
k_D(z_n, \gamma_1(t)-k_D(1, \gamma_1(t) )&\leq 2^{n+2}|y|-\frac{9}{8}3^{n}+\frac{(n+1)\ln2}{2}+\frac{5}{8}=:T_n.
\end{split}
\end{equation*}
Hence, for every $n$ fixed, we have
\[
\lim_{t\ra +\infty}[k_D(z_n, \gamma_1(t)-k_D(1, \gamma_1(t) )]\leq T_n.
\]
Since $\lim_{n\to\infty}T_n=-\infty$, taking $n_M\in\N$ such that $T_n<M$ for all $n\geq n_M$,  \eqref{Eq:stay-in-oro-def-ex}   holds and we are done.
\end{proof}

\section{Simply connected domains whose prime ends have singleton principal part }\label{sub:easy}

Let $\Omega\subsetneq \C$ be a simply connected domain which does {\sl not} have the Denjoy-Wolff Property. This means that there exists a holomorphic self-map of $\Omega$ without fixed points such that the sequence of iterates does not converge to a point. However, it might happen that other elements of  $\hbox{Hol}_d(\Omega,\Omega)$ do have a Denjoy-Wolff point. Is it possible to identify specific conditions which guarantee that a given element of $\hbox{Hol}_d(\Omega,\Omega)$ has a Denjoy-Wolff point? The aim of this section is to focus on such a question. 

We start with a simple fact (compare with \cite[Proposition~5.1]{Ka}):

\begin{lemma}\label{Lem:DW-point-Cara}
Let $\Omega\subsetneq \C$ be a simply connected domain. Let $f\in \hbox{Hol}_d(\Omega,\Omega)$. Then there exists a unique prime end $\underline{\zeta}_f\in\partial_C\Omega$ such that for every $z\in \Omega$ the sequence $\{f^{\circ n}(z)\}$ converges to $\underline{\zeta}_f$ in the Carath\'eodory topology of $\Omega$.
\end{lemma} 
\begin{proof}
The result follows  directly from \cite[Proposition~5.1]{Ka}, after realizing that the identity map of $\Omega$ gives a homeomorphism between the Carath\'eodory compactification of $\Omega$ and the Gromov compactification of $\Omega$ with respect to the hyperbolic distance. However, we give a simple direct proof which does not involve the Gromov compactification.

Let $h:\D\to\Omega$ be a Riemann map.  Let $g:=h^{-1}\circ f\circ h\in \hbox{Hol}_d(\D,\D)$. Hence, by the Denjoy-Wolff Theorem, for every $w\in \D$, the sequence $\{g^{\circ n}(w)\}$ converges to the Denjoy-Wolff point $\tau\in\partial \D$ of $g$. Since the identity map extends as a homeomorphism from the Carath\'eodory compactification $\overline{\D}^C$ of $\D$ (endowed with the Carath\'eodory topology) and the Euclidean closure $\overline{\D}$, it follows that  $\{g^{\circ n}(w)\}$ converges to $\underline{\tau}\in\partial_C\D$ in the Carath\'eodory topology of $\D$. Since $\hat{h}:\overline{\D}^C\to \overline{\Omega}^C$ is an homeomorphism, it turns out that $\{f^{\circ n}(g(w))\}$ converges to $\underline{\zeta}:=\hat{h}(\underline{\tau})$ in the Carath\'eodory topology of $\Omega$, and we are done.
\end{proof}

A first simple consequence of Lemma~\ref{Lem:DW-point-Cara} is the following:

\begin{corollary}\label{Cor:simple-DW-impre}
Let $\Omega\subsetneq \C$ be a simply connected domain. Let $\underline{\zeta}\in\partial_C\Omega$ and assume that there exists $p\in\partial\Omega$  such that $I(\underline{\zeta})=\{p\}$. Let $f\in \hbox{Hol}_d(\Omega,\Omega)$ and let $\underline{\zeta}_f\in\partial_C\Omega$ be given by Lemma~\ref{Lem:DW-point-Cara}. If  $\underline{\zeta}_f=\underline{\zeta}$ then $f$ has Denjoy-Wolff point $p$. 
\end{corollary}
\begin{proof}
Let $h:\D\to\Omega$ be a Riemann map.  Let $g:=h^{-1}\circ f\circ h\in \hbox{Hol}_d(\D,\D)$. Since $\hat{h}:\overline{\D}^C\to \overline{\Omega}^C$ is a homeomorphism, it follows that $\{g^{\circ n}(z)\}$ converges to $\hat{h}^{-1}(\underline{\zeta})$ in the Carath\'eodory topology of $\D$ for every $z\in\D$. Thus, if $\sigma\in\partial\D$ is such that $\underline{\sigma}=\hat{h}^{-1}(\underline{\zeta})$, we have that $\{g^{\circ n}(z)\}$ converges to $\sigma$. Since $I(\hat{h}(\underline{\sigma}))=I(\underline{\zeta})=\{p\}$, it follows that $h$ has (unrestricted) limit $p$ at $\sigma$ (see Subsection~\ref{sec:limits}). But $f^{\circ n}(h(z))=h(g^{\circ n}(z))$, thus $\{f^{\circ n}(h(z))\}$ converges to $p$, and we are done.
\end{proof}

On the opposite direction we have: 
\begin{proposition}\label{Prop:no-DW}
Let $\Omega\subsetneq\C$ be a simply connected domain. Let $\underline{\zeta}\in\partial_C\Omega$ be such that $\Pi(\underline{\zeta}\}$ is not a singleton. Let $f\in \hbox{Hol}_d(\Omega,\Omega)$ and let $\underline{\zeta}_f\in\partial_C\Omega$ be given by Lemma~\ref{Lem:DW-point-Cara}. If  $\underline{\zeta}_f=\underline{\zeta}$ then $f$ does not have Denjoy-Wolff point.
\end{proposition}

The proof is based on the following ``sequential'' Lehto-Virtanen's Theorem''  which was firstly established in \cite{ArBr} for mapping from a ball to a ball (see \cite[Theorem~1.5]{ArBr}):

\begin{lemma}[Sequential Lehto-Virtanen's Theorem for Riemann maps]\label{Lem:LV}
Let $h:\D\to \C$ be univalent and let $\sigma\in\partial\D$. Suppose $\{z_n\}\subset\D$ is a sequence converging to $\sigma$ such that there exists $C>0$ such that for al $n\in\N$
\[
k_\D(z_n,z_{n+1})\leq C.
\]
If there exists $p\in\C_\infty$ such that $\{h(z_n)\}$ converges to $p$ then $h$ has non-tangential limit $p$ at $\sigma$, that is, $\angle\lim_{z\to \sigma}h(z)=p$.
\end{lemma}
\begin{proof}
Although stated n a different context, the result follows essentially from the proof of \cite[Theorem~1.5]{ArBr}, so we just sketch it here. Consider the continuous curve $\gamma:[0,+\infty)\to \D$ defined as 
\[
\gamma(s+n)=(1-s)z_n+sz_{n+1}, \quad n\in\N, s\in [0,1).
\]
We claim that $\lim_{t\to+\infty}\gamma(t)=\sigma$. Indeed, since $k_\D$ is convex, we have
\[
k_\D(\gamma(s+n), z_n)=k_\D((1-s)z_n+sz_{n+1}, z_n)\leq \max\{k_\D(z_n, z_n), k_\D(z_{n+1},z_n)\}\leq C.
\]
Hence, for every subsequence $\{t_k\}$ of positive real numbers converging to $\infty$, we have that $\{\gamma(t_k)\}$ stays at finite hyperbolic distance from $\{z_n\}_{n\geq n_0}$ for all $n_0\in\N$. The only possibility is that $\{\gamma(t_k)\}$ converges to $\sigma$.

Now, since $h$ is univalent, hence proper, it follows that $p\in\partial\Omega$. Since $h$ is an isometry between $k_\D$ and $k_\Omega$, it follows that for every subsequence $\{t_k\}$ of positive real numbers converging to $\infty$ the sequence $\{h(\gamma(t_k))\}$ stays at finite hyperbolic distance from $\{h(z_n)\}_{n\geq n_0}$ for all $n_0\in\N$. Since $\{h(z_n)\}$ converges to $p\in\partial\Omega$, by the so-called ``distance lemma'' (see for example \cite[Theorem~5.3.1]{BCDbook}), $\{h(\gamma(t_k))\}$ has to converge to $p$ as well. That is, $\lim_{t\to+\infty}h(\gamma(t))=p$. 

The result follows then from the classical Lehto-Virtanen's Theorem (see, {\sl e.g.}, \cite[Theorem~3.3.1]{BCDbook}).
\end{proof}

\begin{proof}[Proof of Proposition~\ref{Prop:no-DW}]
Let $h:\D\to\Omega$ be a Riemann map.  Let $f\in \hbox{Hol}_d(\Omega,\Omega)$ and let $\underline{\zeta}_f\in\partial_C\Omega$. Assume by contradiction that there exist $p\in\partial\Omega$ and $z\in\D$ such that $\{f^{\circ n}(h(z))\}$ converges to $p$.

Let $g:=h^{-1}\circ f\circ h\in \hbox{Hol}_d(\D,\D)$. Then  the sequence $\{g^{\circ n}(z)\}$ converges to $\underline{\sigma}:=\hat{h}^{-1}(\underline{\zeta})$ in the Carath\'eodory topology of $\D$, for some $\sigma\in\partial\D$. Thus, $\{g^{\circ n}(z)\}$ converges (in the Euclidean topology) to $\sigma$, since the identity map extends as a homeomorphism from the Carath\'eodory closure of $\D$ to the Euclidean closure of $\D$. But, for every $n\in\N$,
\[
k_\D(g^{\circ n}(z), g^{\circ (n+1)}(z))\leq C:=k_\D(z,g(z)).
\]
Since $h(g^{\circ n}(z))=f^{\circ n}(h(z))$, by Lemma~\ref{Lem:LV}, $h$ has non-tangential limit $p$ at $\sigma$. In particular, the non-tangential cluster set $\Gamma_N(h,\sigma)=\{p\}$.  Since $\hat{h}(\underline{\sigma})=\underline{\zeta}$ we have $\Pi(\underline{\zeta})=\Gamma_N(h,\sigma)=\{p\}$, a contradiction.
\end{proof}

We end this section with some positive results for a large class of holomorphic self-maps without fixed points.  Recall the definition of divergence rate $(cf)$ from Subsection~\ref{subsect:divergence}.

\begin{definition}
Let $\Omega\subsetneq \C$ be a simply connected domain and let $f\in \hbox{Hol}_d(\Omega,\Omega)$. We say that $f$ is {\sl hyperbolic} if $c(f)>0$.
\end{definition}

For hyperbolic maps, the condition about impressions of prime ends in Corollary~\ref{Cor:simple-DW-impre} can be weaken considering instead the principal parts:

\begin{proposition}\label{Prop:hyper-DW-princ}
Let $\Omega\subsetneq \C$ be a simply connected domain. Let $\underline{\zeta}\in\partial_C\Omega$ and assume that there exists $p\in\partial\Omega$  such that $\Pi(\underline{\zeta})=\{p\}$. Let $f\in \hbox{Hol}_d(\Omega,\Omega)$ be hyperbolic and let $\underline{\zeta}_f\in\partial_C\Omega$ be given by Lemma~\ref{Lem:DW-point-Cara}. If  $\underline{\zeta}_f=\underline{\zeta}$ then $f$ has Denjoy-Wolff point $p$. 
\end{proposition}
\begin{proof}
Let $h:\D\to\Omega$ be a Riemann map and $\sigma\in\partial\D$ such that $\hat{h}(\underline{\sigma})=\underline{\zeta}$. By hypothesis, $\Pi(\underline{\zeta})=\{p\}$, hence the non-tangential cluster set of $h$ at $\sigma$ is $\{p\}$ and thus $h$ has non-tangential limit $p$ at $\sigma$. In particular, if $\{z_n\}\subset\D$ is any sequence converging to $\sigma$ non-tangentially, it follows that $\{h(z_n)\}$ converges to $p$. 

Let $g:=h^{-1}\circ f\circ h\in \hbox{Hol}_d(\D,\D)$. Clearly, $c(g)=c(f)>0$ since $h$ is an isometry for the hyperbolic distance. As before, since $\{f^{\circ n}(h(z))\}$ converges to $\underline{\zeta}$ in the Carath\'eodory topology of $\Omega$ for all $z\in \D$, it follows that $\{g^{\circ n}(z)\}$ converges to $\underline{\sigma}$ in the Carath\'eodory topology of $\D$ for all $z\in \D$ and hence $\{g^{\circ n}(z)\}$ converges to $\sigma$---that is, $\sigma$ is the Denjoy-Wolff point of $g$. 

If $\{g^{\circ n}(z)\}$ converges to $\sigma$ non-tangentially, since $g^{\circ n}(z)=h(f^{\circ n}(h(z)))$, it follows that $\{f^{\circ n}(h(z))\}$ converges to $p$, and we are done.

Therefore, we have to prove that $\{g^{\circ n}(z)\}$ converges to $\sigma$ non-tangentially for every $z\in\D$. By \cite[Proposition~5.8]{ArBr}, $c(g)=-\log \lambda_g$, where 
\[
\lambda_g:=\liminf_{z\to \sigma}\frac{1-|g(z)|}{1-|z|}=\angle\lim_{z\to\sigma}g'(z).
\]
 Since $c(g)>0$, it follows that $\lambda_g\in (0,1)$. Hence, $g$ is a hyperbolic self-map of $\D$ (in the classical sense). Therefore, $\{g^{\circ n}(z)\}$ converges to $\sigma$ non-tangentially for every $z\in\D$ by a result of Cowen \cite[Lemma~2.1]{Co}  (see also \cite[Proposition~1.8.7]{BCDbook}).
\end{proof}

As a corollary we have
\begin{corollary}\label{Cor:hyper-DW}
Let $\Omega\subsetneq \C$ be a simply connected domain such that the principal part of every prime end is a point. Then  every hyperbolic $f\in \hbox{Hol}_d(\Omega,\Omega)$ has Denjoy-Wolff point. 
\end{corollary}

\end{document}